\documentclass[11pt]{amsart}
\usepackage[margin=1.35in]{geometry}
\usepackage{amscd,amsmath,amsxtra,amsthm,amssymb,stmaryrd,xr,mathrsfs,mathtools,enumerate,commath, comment}
\usepackage{stmaryrd}
\usepackage{xcolor}
\usepackage{commath}
\usepackage{comment}
\usepackage{tikz-cd}
\usepackage{longtable} 
\usepackage{pdflscape} 
\usepackage{booktabs}
\usepackage{hyperref}
\definecolor{vegasgold}{rgb}{0.77, 0.7, 0.35}
\definecolor{darkgoldenrod}{rgb}{0.72, 0.53, 0.04}
\definecolor{gold(metallic)}{rgb}{0.83, 0.69, 0.22}
\hypersetup{
 colorlinks=true,
 linkcolor=darkgoldenrod,
 filecolor=brown,      
 urlcolor=gold(metallic),
 citecolor=darkgoldenrod,
 pdftitle={Diophantine stability for elliptic curves on average},
 }

\usepackage[all,cmtip]{xy}

\DeclareFontFamily{U}{wncy}{}
\DeclareFontShape{U}{wncy}{m}{n}{<->wncyr10}{}
\DeclareSymbolFont{mcy}{U}{wncy}{m}{n}
\DeclareMathSymbol{\Sh}{\mathord}{mcy}{"58}
\usepackage[T2A,T1]{fontenc}
\usepackage[OT2,T1]{fontenc}

\newtheorem{theorem}{Theorem}[section]
\newtheorem{lemma}[theorem]{Lemma}

\newtheorem*{ass*}{Assumption}
\newtheorem{definition}[theorem]{Definition}
\newtheorem{corollary}[theorem]{Corollary}

\newtheorem{conjecture}[theorem]{Conjecture}
\newtheorem{proposition}[theorem]{Proposition}

\newtheorem{lthm}{Theorem}

\newcommand{\cB}{\mathcal{B}}

\newcommand{\cF}{\mathcal{F}}

\newcommand{\cT}{\mathcal{T}}

\newcommand{\Z}{\mathbb{Z}}

\newcommand{\Q}{\mathbb{Q}}
\newcommand{\F}{\mathbb{F}}

\newcommand{\cC}{\mathcal{C}}

\newcommand{\cO}{\mathcal{O}}

\newcommand{\op}[1]{\operatorname{#1}}

\newcommand\mtx[4] { \left( {\begin{array}{cc}
 #1 & #2 \\
 #3 & #4 \\
 \end{array} } \right)}

\numberwithin{equation}{section}

\begin{document}

\title[Diophantine stability for elliptic curves on average]{Diophantine stability for elliptic curves on average}

\author[A.~Ray]{Anwesh Ray}
\thanks{corresponding author: Anwesh Ray, email: ar2222@cornell.edu}
\address[Ray]{Chennai Mathematical Institute, H1, SIPCOT IT Park, Kelambakkam, Siruseri, Tamil Nadu 603103, India}
\email{anwesh@cmi.ac.in}

\author[T.~Weston]{Tom Weston}
\address[Weston]{Department of Mathematics, University of Massachusetts, Amherst, MA, USA.} 
 \email{weston@math.umass.edu}


\maketitle

\begin{abstract}
Let $K$ be a number field and $\ell \geq 5$ a prime number.  
Mazur and Rubin introduced the notion of \emph{diophantine stability} for a variety $X_{/K}$ at a prime $\ell$. 
We show that there is a positive density set of elliptic curves $E_{/\mathbb{Q}}$ of rank $1$ such that $E_{/K}$ is diophantine stable at $\ell$. 
This has implications for Hilbert's Tenth Problem over $\cO_K$. 
This problem asks whether there exists an algorithm that decides in finite time whether a finite system of Diophantine equations over $\cO_K$ has a solution.
\end{abstract}

\section{Introduction}
\par Let $V$ be a variety over a number field $K$ and let $\ell$ be a prime number. Mazur and Rubin introduced the notion of diophantine stability. Given a number field extension $L/K$, the variety $V_{/K}$ is \emph{diophantine stable} in $L$ if $V(L)=V(K)$. For a prime $\ell$ it is said that $V$ is \emph{$\ell$-diophantine stable} over $K$ if for every $n\in \Z_{\geq 1}$ and finite set of primes $\Sigma$ of $K$, there are infinitely many cyclic extensions $L/K$ of degree $\ell^n$, such that 
\begin{enumerate}
    \item all primes $v\in \Sigma$ are completely split in $L$,
    \item $V(L)=V(K)$.
\end{enumerate}

\begin{theorem}[Mazur-Rubin \cite{mazur2018diophantine}, Theorem 1.2]Let $A_{/K}$ be a simple abelian variety for which all geometric endomorphisms are defined over $K$. Then, there is a set of prime numbers $S$ of positive density such that $A$ is $\ell$-diophantine stable for all $\ell\in S$.
\end{theorem}
When specialized to elliptic curves, this result has significant consequences to Hilbert's Tenth Problem for number rings.
\begin{corollary}[Mazur-Rubin \cite{mazur2018diophantine}, Corollary 1.6]
    For every prime $\ell$, there are uncountably many pairwise non-isomorphic 
totally real fields $L$ of 
algebraic numbers in $\Q_\ell$ over which the following two statements both hold:
\begin{enumerate} 
\item 
There is a diophantine definition of $\Z$ in the ring of integers $\cO_L$ of $L$. 
In particular, Hilbert's Tenth Problem has a negative answer for $\cO_L$; 
i.e., there does not exist an algorithm to determine whether a polynomial 
(in many variables) with coefficients in $\cO_L$ has a solution in $\cO_L$.
 \item 
There exists a first-order definition of the ring $\Z$ in $L$. 
The first-order theory for such fields $L$ is undecidable.
\end{enumerate}

\end{corollary}

\subsection{Main results} We study statistical questions from a different perspective. Any elliptic curve $E_{/\Q}$ is isomorphic to a unique curve of the form \[E=E_{A,B}:y^2=x^3+Ax+B,\] where $(A,B)\in \Z^2$ such that for all primes $p$, either $p^4\nmid A$ or $p^6\nmid B$, cf. \cite[p. 814, l.-5]{duke1997elliptic}. Such a Weierstrass equation is minimal and the (naive) height of $E$ is defined as follows
\[H(E)=H(E_{A, B}):=\op{max}\{|A|^3, B^2\}.\] Here, we follow the convention in \emph{loc. cit.}. It is also customary to define the height to be $\op{max}\{4|A|^3, 27B^2\}$, however, this is only a matter of convention.
Let $\mathcal{C}$ be the set of all isomorphism classes of elliptic curves over $\Q$ and $\mathcal{C}(X)$ those with height $\leq X^6$. As is well known, \[\#\mathcal{C}(X)=C_1 X^5+O(X^3),\] where $C_1=\frac{4}{\zeta(10)}$, cf. \cite[Lemma 4.3]{brumer1992average}.
Let $\mathcal{S}$ be a subset of isomorphism classes of elliptic curves defined over $\Q$. Set $\mathcal{S}(X):=\mathcal{S}\cap \mathcal{C}(X)=\{E\in \mathcal{S}\mid H(E)\leq X^6\}$. The density of $\mathcal{S}$ is given by
\[\delta(\mathcal{S}):=\lim_{X\rightarrow \infty} \frac{\# \mathcal{S}(X)}{\# \mathcal{C}(X)},\] provided the limit exists.
\par In any case, we can define the upper and lower densities by \[\overline{\delta}(\mathcal{S}):=\limsup_{X\rightarrow \infty} \frac{\# \mathcal{S}(X)}{\# \mathcal{C}(X)}\] 
and \[\underline{\delta}(\mathcal{S}):=\liminf_{X\rightarrow \infty} \frac{\# \mathcal{S}(X)}{\# \mathcal{C}(X)}\] respectively. We prove the following result.
\begin{lthm}[Theorem \ref{our main result}]\label{main thm 1}
    Let $\ell\geq 5$ be a prime number and let $K$ be a number field. Then, the set of elliptic curves $E_{/\Q}$ that are $\ell$-diophantine stable over $K$ has density $1$.
\end{lthm}
 We recall a result of Bhargava and Skinner, which will be shown to have consequences for Hilbert's Tenth Problem over number rings.

\begin{theorem}[Bhargava-Skinner \cite{BhargavaSkinner}]\label{BS thm}
    The lower density of elliptic curves $E_{/\Q}$ with rank $1$ is positive.
\end{theorem}

\begin{lthm}[Theorem \ref{main thm 2.6}]\label{main thm pos density}
    Let $\ell\geq 5$ be a prime number and let $K$ be a number field. Then, the set of elliptic curves $E_{/\Q}$ for which
    \begin{enumerate}
        \item $E_{/K}$ is $\ell$ diophantine stable, 
        \item $\op{rank}E(\Q)=1$
    \end{enumerate}
    has positive lower density.
\end{lthm}

\begin{lthm}\label{main thm 2}
    Let $\ell\geq 5$ be a prime number, let $K$ be a number field, $n\in \Z_{\geq 1}$ and $\Sigma$ be a finite set of primes of $K$. Assume that $\Z$ is a diophantine subset of $\cO_K$, and consequently, Hilbert's Tenth Problem has a negative answer for $\cO_K$. Then, there are infinitely many degree $\ell^n$ cyclic extensions $L/K$ in which the primes of $\Sigma$ are completely split, such that $\Z$ is diophantine in $\cO_L$ and Hilbert's Tenth Problem has a negative answer for $\cO_L$.
\end{lthm}
The key point is that the above result holds for all primes $\ell\geq 5$. It follows from results of Shlapentokh and Mazur-Rubin that the above assertion holds for a density $1$ set of primes $\ell$.

\subsection{Organization} Including the introduction, the article consists of four sections. In section \ref{s 2}, we recall a criterion of Mazur and Rubin for diophantine stability, cf. Theorem \ref{mazur rubin criterion}. For elliptic curves $E_{/\Q}$, this criterion can be verified if the field cut out by the residual representation satisfies some additional conditions, cf. Proposition \ref{T K l prop}. Density results are proven in section \ref{s 3}. It is shown that the conditions of Theorem \ref{mazur rubin criterion} are satisfied for a set of elliptic curves of density $1$. Theorems \ref{main thm 1} and \ref{main thm pos density} are proven at the end of this section. In section \ref{s 4}, we discuss applications to Hilbert's Tenth Problem and give the proof of Theorem \ref{main thm 2}, which is seen to follow from Theorem \ref{main thm pos density} and a well known result of Shlapentokh.

\subsection{Subsequent developments}

This paper was announced on the arXiv in April of 2023, and its ideas have since proven fruitful. In December 2024, Koymans and Pagano \cite{koymans2024hilbert} announced a proof of Hilbert's Tenth Problem for rings of integers of number fields, followed in January 2025 by a second proof, due to Alp\"oge, Bhargava, Ho, and Shnidman \cite{alpoge2025rank}. Both works rely crucially on recent additive combinatorics results of Kai \cite{kai2023linear}. The results in \cite{koymans2024hilbert} relies on a Markov model governing the dimensions of $2$-Selmer groups of quadratic twist families of elliptic curves over number fields. The strategy in \cite{alpoge2025rank} applies additive combinatorics results to a twist family of Jacobians of certain genus two hyperelliptic curves. The analytic tools and statistical viewpoint developed in this manuscript are quite different in nature. Our approach crucialy relies on a theorem of Duke \cite{duke1997elliptic}, which roughly states that almost all elliptic curves over $\Q$ (when ordered by height) have large Galois images associated to them. Given the strategies pursued in both works, it is reasonable to suggest that the perspective introduced here (studying rank stability through a statistical approach) has played a significant yet indirect role in shaping the broader context of these subsequent developments. In addition, our results offer a complementary direction for understanding Diophantine stability phenomena from a statistical perspective.
\subsection*{Acknowledgments} We would like to thank the referees for their careful review of the manuscript and for their valuable comments and corrections. The project was completed while the first-named author was a CRM–Simons postdoctoral fellow, and he gratefully acknowledges support from Simons. We are also indebted to Lea Beneish, Barry Mazur, and Ravi Ramakrishna for their helpful suggestions.

\subsection*{Data Availability} There is no data associated to the results of this manuscript.

\section{Diophantine stability}\label{s 2}
\subsection{Diophantine stability for elliptic curves}
\par We review the notion of diophantine stability for an elliptic curve. Throughout, we shall fix a number field $K$ and a prime number $\ell\geq 5$. For $n\in \Z_{\geq 1}$, we say that an extension $L$ of $K$ is a 
$\Z/\ell^n\Z$-extension if it is a Galois extension of $K$ such that $\op{Gal}(L/K)$ is isomorphic to $\Z/\ell^n \Z$. We recall the notion of diophantine stability, due to Mazur and Rubin \cite{mazur2018diophantine}.

\begin{definition}Let $L/K$ be a field extension and let $E_{/K}$ be an elliptic curve. Then, $E_{/K}$ is said to be diophantine stable in $L$ if $E(L)=E(K)$. It is said that $E_{/K}$ is diophantine stable at $\ell$ if for all $n\in \Z_{\geq 1}$ and every finite set of primes $\Sigma$ of $K$, there are infinitely many $\Z/\ell^n \Z$-extensions $L/K$ such that 
\begin{enumerate}
    \item $E(L)=E(K)$,
    \item all primes in $\Sigma$ are completely split in $L$.
\end{enumerate} 
Given an elliptic curve $E_{/\Q}$, we say that $(E,K, \ell)$ satisfies (DS) if $E_{/K}$ is diophantine stable at $\ell$.
\end{definition}
Given a number field $K$, set $\op{G}_K$ to denote the absolute Galois group $\op{Gal}(\bar{K}/K)$. Let $E_{/\Q}$ be an elliptic curve and denote by $E[\ell]$, the $\ell$-torsion group 
\[E[\ell]:=\op{ker}\left(E(\bar{\Q})\xrightarrow{\times \ell} E(\bar{\Q})\right).\]
Set $\bar{\rho}_E=\bar{\rho}_{E, \ell}:\op{G}_{\Q}\rightarrow \op{GL}_2(\F_\ell)$ to denote the Galois representation on $E[\ell]$. Set $\Q(E[\ell])$ to be the field cut out by $E[\ell]$, i.e., $\Q(E[\ell]):=\bar{\Q}^{\op{ker}\bar{\rho}_{E, \ell}}$. 

\par We recall a criterion for diophantine stability, specialized to rational elliptic curves without complex multiplication.

\begin{theorem}[Mazur-Rubin]\label{mazur rubin criterion}
Let $K$ be a number field, $E_{/\Q}$ an elliptic curve without complex multiplication, and $\ell\geq 3$ a prime number.  Assume that the following conditions hold
\begin{enumerate}
    \item $E[\ell]$ is an irreducible $\op{G}_K$-module, 
    \item $H^1(K(E[\ell])/K, E[\ell])=0$,
    \item there is no abelian extension of degree $\ell$ of $K(\mu_\ell)$ contained in $K(E[\ell])$,
    \item there is $\tau_0\in \op{G}_{K(\mu_\ell)}$ such that $E[\ell]/(\tau_0-1)E[\ell]=0$,
    \item there is $\tau_1\in \op{G}_{K(\mu_\ell)}$ such that $\op{dim}_{\F_\ell}\left(E[\ell]/(\tau_1-1)E[\ell]\right)=1$.
\end{enumerate}
Then, $E_{/K}$ is diophantine stable at $\ell$.
\end{theorem}

\begin{proof}
    The result follows from \cite[Theorem 9.21]{mazur2018diophantine}.
\end{proof}

 In the next section, we prove the following result.

\begin{theorem}\label{our main result}
    Let $\ell\geq 5$ be a prime number and let $K$ be a number field. Then, the set of elliptic curves $E_{/\Q}$ for which the conditions of Theorem \ref{mazur rubin criterion} are satisfied has density $1$. 
\end{theorem}

\begin{definition}\label{s k l def}
    Let $\mathcal{S}_{K, \ell}$ be the set of elliptic curves $E_{/\Q}$ such that 
\begin{enumerate}
    \item $\op{rank}E(\Q)=1$,
    \item $(E,K,\ell)$ satisfies (DS).
\end{enumerate}
\end{definition}

\begin{theorem}\label{main thm 2.6}
    Let $\ell\geq 5$ be a prime number and $K$ be any number field. Then, $\mathcal{S}_{K, \ell}$ has positive lower density.
\end{theorem}
\begin{proof}
    This is a direct consequence of Theorem \ref{our main result} and Theorem \ref{BS thm}.
\end{proof}

\subsection{A criterion for diophantine stability in terms of the residual image} In the remainder of this section, we shall introduce some further notation and establish a criterion for the conditions of Theorem \ref{mazur rubin criterion} to be satisfied. We fix $(K, \ell)$ where $K$ is a number field and $\ell\geq 5$ is a prime number. We let $E_{/\Q}$ be an elliptic curve and recall that \[\bar{\rho}_E:\op{G}_{\Q}\rightarrow \op{GL}_2(\F_\ell)\] denotes the representation on $E[\ell]$. Set $\mathbf{1}$ to denote the identity element of $\op{GL}_2(\F_\ell)$ and $\langle -\mathbf{1}\rangle$ the subgroup of order $2$ generated by $-\mathbf{1}$. We set $\op{GL}_2'(\F_\ell):=\frac{\op{GL}_2(\F_\ell)}{\langle -\mathbf{1}\rangle}$ and denote by \[\bar{\rho}_E':\op{G}_{\Q}\rightarrow \op{GL}_2'(\F_\ell)\] the homomorphism which is the composite of $\bar{\rho}_E$ with the natural quotient map. The determinant character $\op{det}\bar{\rho}_E$ is the mod-$\ell$ cyclotomic character $\chi_\ell$, and thus the kernel of $\op{det}\bar{\rho}_E$ is $\op{G}_{\Q(\mu_\ell)}$. Note that $\langle -\mathbf{1}\rangle$ lies in the kernel of the determinant map, and hence, $\op{det}:\op{GL}_2(\F_\ell)\rightarrow \F_\ell^\times$ factors through $\op{GL}_2'(\F_\ell)$. Note that the restriction of $\bar{\rho}_E$ and $\bar{\rho}_E'$ to $\op{G}_{\Q(\mu_\ell)}$ gives rise to homomorphisms 
$\bar{\rho}_{E}:\op{G}_{\Q(\mu_\ell)}\rightarrow \op{SL}_2(\F_\ell)$ and $\bar{\rho}_{E}':\op{G}_{\Q(\mu_\ell)}\rightarrow \op{PSL}_2(\F_\ell)$ respectively. Galois \cite[p.\ 412]{galois1846works}, in a letter to Chevalier showed that when $\ell\geq 5$, the group $\op{PSL}_2(\F_\ell)$ is simple. This property shall prove to be considerably useful in establishing our results.

\par Given a number field $F$, a finite group $\mathcal{G}$, and a homomorphism $\varrho: \op{G}_F\rightarrow \mathcal{G}$, we let $F(\varrho)$ denote the field $\bar{F}^{\op{ker}\varrho}$. We identify the Galois group $\op{Gal}(F(\varrho)/F)$ with the image of $\varrho$. We refer to $F(\varrho)$ as the extension of $F$ cut out by $\varrho$. In particular, when $\varrho$ is surjective, $F(\varrho)$ is a $\mathcal{G}$-extension of $F$. With respect to the above notation, we set $F(\bar{\rho}_E)$ (resp. $F(\bar{\rho}_E')$) to be the extension of $F$ cut out by $\bar{\rho}_E$ (resp. $\bar{\rho}_E'$). We shall also write $F(E[\ell])$ to denote the field $F(\bar{\rho}_E)$. Let $G_E$ (resp. $G_E'$) denote the Galois group $\op{Gal}(\Q(\bar{\rho}_E)/\Q)$ (resp. $\op{Gal}(\Q(\bar{\rho}_E')/\Q)$). Set $D:=\op{Gal}(\Q(\bar{\rho}_E)/\Q(\bar{\rho}_E'))$ and identify $G_E'$ with $G_E/D$. The residual representation $\bar{\rho}_E$ induces an injection \[\bar{\rho}_E:G_E\hookrightarrow \op{GL}_2(\F_\ell),\] and $D$ lies in the kernel of the determinant character
\begin{equation}\label{det character}\op{det}\bar{\rho}_E: G_E\hookrightarrow \op{GL}_2(\F_\ell)\rightarrow \F_\ell^\times.\end{equation} The character $G_E\xrightarrow{\op{det}\bar{\rho}_E} \F_\ell^\times$ \eqref{det character} thus factors through the quotient $G_E'$. Let $H_E$ (resp. $\bar{H}_E$) denote the kernel of $G_E\xrightarrow{\op{det}\bar{\rho}_E}\F_\ell^\times$ (resp. $G_E'\xrightarrow{\op{det}\bar{\rho}_E}\F_\ell^\times$). Note that $D$ is contained in $H_E$ and $\bar{H}_E=H_E/D$. We find that $H_E=\op{Gal}(\Q(\bar{\rho}_E)/\Q(\mu_\ell))$, while $\Q(\mu_\ell)$ is contained $\Q(\bar{\rho}_E')$ and $\bar{H}_E=\op{Gal}(\Q(\bar{\rho}_E')/\Q(\mu_\ell))$. Set $\tilde{K}$ to be the Galois closure of $K$ over $\Q$.

\begin{definition}\label{def T K L}
    Let $\mathcal{T}_{K, \ell}$ be the set of elliptic curves $E_{/\Q}$ such that the following conditions are satisfied
\begin{enumerate}
    \item $\bar{\rho}_E'$ is surjective,
    \item $\tilde{K}(\mu_\ell)$ does not contain $\Q(\bar{\rho}_E')$. 
\end{enumerate}
\end{definition}
Note that if $\bar{\rho}'_E$ is surjective, then the image of $\bar{\rho}_E$ is not solvable. In particular, $E$ does not have complex multiplication.

\begin{lemma}\label{lemma 3.6}
For $E\in \mathcal{T}_{K, \ell}$, \[\op{Gal}(K(\bar{\rho}_E')/K(\mu_\ell))\simeq \op{PSL}_2(\F_\ell).\]
\end{lemma}
\begin{proof}
    Note that since $\bar{\rho}_E'$ is surjective, we find that $\op{Gal}(\Q(\bar{\rho}_E')/\Q(\mu_\ell))\simeq \op{PSL}_2(\F_\ell)$. On the other hand, since $\op{PSL}_2(\F_\ell)$ is simple, and $\tilde{K}(\mu_\ell)$ does not contain $\Q(\bar{\rho}_E')$, it follows that $\tilde{K}(\mu_\ell)\cap \Q(\bar{\rho}_E')=\Q(\mu_\ell)$. This implies that $K(\mu_\ell)\cap \Q(\bar{\rho}_E')=\Q(\mu_\ell)$ and therefore we conclude that 
    \[\op{Gal}(K(\bar{\rho}_E')/K(\mu_\ell))\simeq \op{Gal}(\Q(\bar{\rho}_E')/\Q(\mu_\ell))\simeq \op{PSL}_2(\F_\ell).\]
\end{proof}

\begin{proposition}\label{T K l prop}
For every elliptic curve $E\in \mathcal{T}_{K,\ell}$, $(E,K, \ell)$ satisfies (DS).
\end{proposition}
\begin{proof}
    It suffices to show that the five conditions of Theorem \ref{mazur rubin criterion} are satisfied. Note that Lemma \ref{lemma 3.6} asserts that $\op{Gal}(K(\bar{\rho}_E')/K(\mu_\ell))\simeq \op{PSL}_2(\F_\ell)$.
\begin{enumerate}
    \item Since $\op{Gal}(K(\bar{\rho}_E')/K(\mu_\ell))\simeq \op{PSL}_2(\F_\ell)$, $E[\ell]$ is irreducible as a $\op{G}_K$-module. 
    \item This part follows from \cite{lawson2017vanishing} or \cite[Lemma 2.2]{prasad2021relating}. We summarize the argument here. Let $G\subseteq \op{GL}_2(\F_\ell)$ denote the image of $\bar{\rho}_{E|\op{G}_K}$. Note that since $\op{Gal}(K(\bar{\rho}_E')/K(\mu_\ell))\simeq \op{PSL}_2(\F_\ell)$, $G$ contains an element of order $\ell$. Clearly, $G$ is not contained in a Borel subgroup. It follows from \cite[Proposition 3.1]{sutherland2016computing} that $G$ contains $\op{SL}_2(\F_\ell)$. In particular, $G$ contains the negative identity $-\mathbf{1}$. Identify $G$ with $\op{Gal}(K(E[\ell])/K)$ and let $\Delta$ be the subgroup of $G$ generated by $-\mathbf{1}$. Inflation-restriction yields an exact sequence
    \[0\rightarrow H^1(G/\Delta, E[\ell]^\Delta)\rightarrow H^1(G, E[\ell])\rightarrow H^1(\Delta, E[\ell]).\] Since $E[\ell]^\Delta=0$, it follows that $H^1(G/\Delta, E[\ell]^\Delta)=0$. Since $D$ has order $2$, its order is coprime to that of $E[\ell]$ and we find that $H^1(\Delta, E[\ell])=0$. Therefore, we conclude from the above exact sequence that $H^1(G, E[\ell])=0$.
    \item Since $\op{PSL}_2(\F_\ell)$ is simple, $\op{Gal}(K(\bar{\rho}_E')/K(\mu_\ell))\simeq \op{PSL}_2(\F_\ell)$ and the degree $[K(E[\ell]):K(\bar{\rho}_E')]$ is prime to $\ell$, it follows that there is no abelian extension of degree $\ell$ of $K(\mu_\ell)$ contained in $K(E[\ell])$.
    \item Since $\bar{\rho}_E'(\op{G}_{K(\mu_\ell)})$ is isomorphic to $\op{PSL}_2(\F_\ell)$, it is clear that there is a diagonal element $\mtx{a}{}{}{a^{-1}}\in \bar{\rho}_E(\op{G}_{K(\mu_\ell)})$, such that $a\neq \pm 1$. Let $\tau_0\in \op{G}_{K(\mu_\ell)}$ be chosen so that $\bar{\rho}(\tau_0)=\mtx{a}{}{}{a^{-1}}$. Then, it follows that $E[\ell]/(\tau_0-1)E[\ell]=0$. 
    \item It is easy to see that $\bar{\rho}_E(\op{G}_{K(\mu_\ell)})$ contains a unipotent element $\mtx{1}{b}{}{1}$, with $b\neq 0$. Let $\tau_1\in \op{G}_{K(\mu_\ell)}$ be such that $\bar{\rho}(\tau_1)=\mtx{1}{b}{}{1}$, we find that \[\op{dim}_{\F_\ell}\left(E[\ell]/(\tau_1-1)E[\ell]\right)=1.\]
\end{enumerate}
\end{proof}

\section{Density results}\label{s 3}

\par Recall from Definition \ref{def T K L} that $\mathcal{T}_{K, \ell}$ is the set of elliptic curves $E_{/\Q}$ such that the following conditions are satisfied
\begin{enumerate}
    \item $\bar{\rho}_E'$ is surjective,
    \item $\tilde{K}(\mu_\ell)$ does not contain $\Q(\bar{\rho}_E')$. 
\end{enumerate}
According to Proposition \ref{T K l prop}, for every elliptic curve $E\in \cT_{K, \ell}$, $(E,K,\ell)$ satisfies (DS). We prove that the set of elliptic curves $\mathcal{T}_{K, \ell}$ has density $1$. Let $\mathcal{T}_\ell'$ be the set of elliptic curves $E_{/\Q}$ for which $\bar{\rho}_E'$ is surjective. We shall assume throughout this section that $\ell\geq 5$. Note that the elliptic curves $E\in \mathcal{T}_\ell'$ do not have complex multiplication. Duke \cite{duke1997elliptic} showed that the set of elliptic curves $E_{/\Q}$ for which $\bar{\rho}_E:\op{G}_{\Q}\rightarrow \op{GL}_2(\F_\ell)$ is surjective has density $1$.  In particular, it follows from Duke's result that the set $\mathcal{T}_\ell'$ has density $1$. The result proven in \emph{loc. cit.} is much stronger. A prime $\ell$ is an \emph{exceptional prime} of an elliptic curve $E_{/\Q}$ if the residual representation at $\ell$ is not surjective. It is proven that a density $1$ set of elliptic curves $E_{/\Q}$ have no exceptional primes. An elliptic curve $E_{/Q}$ is said to be a \emph{Serre curve} if the image of its associated adelic Galois representation has index $2$ in $\op{GL}_2(\widehat{\Z})$. Jones \cite{jones2010almost} proves that the set of Serre curves has density $1$.

\par Let $A_{/\Q}$ be an elliptic curve and let $\mathcal{T}_A$ be the set of elliptic curves $E_{/\Q}$ such that $\bar{\rho}_E'\simeq \bar{\rho}_A'$. Given an elliptic curve $E$, we write $N_E$ for its conductor. At a prime $p$, let $\op{G}_p$ denote the absolute Galois group $\op{Gal}(\bar{\Q}_p/\Q_p)$ and $\sigma_p\in \op{G}_p$ be a lift of the Frobenius element. For $p\nmid N_E\ell$, \[\op{trace}\bar{\rho}_E(\sigma_p) \equiv a_p(E)\pmod{\ell};\] set $t_p(E)\in \F_\ell$ to denote the trace of $\bar{\rho}_E(\sigma_p)$ modulo $\ell$. For $E\in \cT_A$, we find that \begin{equation}\label{t_p(A)=pm t_p(E)}t_p(E)=\pm t_p(A)\end{equation} for all primes $p\nmid N_EN_A\ell$.
\par Let $E_1$ and $E_2$ be elliptic curves over $\Q$ and let $(t_1, t_2, d)\in \F_\ell^3$ be such that $d\neq 0$. Let $\pi(X, d, \ell)$ be the number of prime numbers $p\leq X$ such that $p\equiv d\pmod{\ell}$, and set $\pi_{E_1, E_2}(X, t_1, t_2, d, \ell)$ to be the number of prime numbers $p\leq X$ such that 
\begin{enumerate}
    \item $p\nmid N_{E_1}N_{E_2}$, 
    \item $p\equiv d\pmod{\ell}$, 
    \item $t_p(E_i)=t_i$ for $i=1,2$.
\end{enumerate}

For $t\in \F_\ell$, set $\chi_{t,d,\ell}$ to denote the Kronecker symbol $\left(\frac{t^2-4d}{\ell}\right)$, set 
\[\delta(t,d, \ell) :=\left(\frac{\ell+\chi_{t,d,\ell}}{\ell^2-1}\right),\]
and set
\begin{equation}\label{def of delta}\delta=\delta(t_1, t_2, d, \ell):=\delta(t_1,d, \ell)\delta(t_2,d, \ell).\end{equation}

For an integer $n$, set $H(n)$ to denote the \emph{Hurwitz class number}. We refer to \cite[p. 816, l. 24]{duke1997elliptic} or \cite[p.293]{cox2022primes} for the definition. Let $E=E_{r,s}$ be the elliptic curve with minimal Weierstrass equation $y^2=x^3+rx+s$, and assume that this equation is minimal. Let $\Delta_E$ be the discriminant of $E$.  The Frobenius trace $a_p(E)$ depends only on $(r,s)$, and
\[a_p(E)=a_{r,s}(p):=-\sum_{x \in \F_p}\left(\frac{x^3+rx+s}{p}\right).\]Let $p\geq 5$ be a prime number. We recall a well known formula of Deuring for the number of elliptic curves over $\F_p$ which have a preassigned number of points.
\begin{theorem}[Deuring]\label{deuring thm}
    Let $p\geq 5$ be a prime and $N=p+1-a$ be an integer such that $|a|<2\sqrt{p}$. Then the number of Weierstrass equations for elliptic curves $E_{/\F_p}$ for which $\#E(\F_p)=N$ is $\left(\frac{p-1}{2}\right)H(4p-a^2)$. In other words, 
    \[\# \left\{(r,s)\in (\Z/p\Z)^2\mid 4r^3+27s^2\neq 0\text{ and }a_{r,s}(p)=a\right\}=\left(\frac{p-1}{2}\right)H(4p-a^2).\]
\end{theorem}
\begin{proof}
    We refer to \cite[Theorem 14.18]{cox2022primes} for a proof of the result.
\end{proof}
Recall from the previous section that $\mathcal{C}$ is the set of all isomorphism classes of elliptic curves over $\Q$ and $\mathcal{C}(X)$ those with height $\leq X^6$. Denote by $\mathcal{C}(X)^2$ the Cartesian product $\mathcal{C}(X)\times \mathcal{C}(X)$ consisting of all pairs of elliptic curves $(E_1, E_2)$ defined over $\Q$. Let $f(X)$ and $g(X)$ be functions of $X$ taking on postive values. We write $f(X)\ll g(X)$, or equivalently, $f(X)=O(g(X))$, to mean that there exists a constant $C>0$, independent of $X$, such that $f(X)\leq C g(X)$. 
\begin{proposition}\label{main prop 3.1}
    Let $(t_1, t_2, d)\in \F_\ell^3$ be such that $d\neq 0$. Then, with respect to notation above, 
    \[\frac{1}{\# \mathcal{C}(X)^2}\sum_{(E_1, E_2)\in \mathcal{C}(X)^2}\left(\pi_{E_1, E_2}(X, t_1, t_2, d, \ell)-\delta\pi(X, d, \ell)\right)^2\ll X.\]  
\end{proposition}
\begin{proof}
    The proof of this result follows from a similar argument to that of \cite[Theorem 2]{duke1997elliptic}. For each prime $p$, let $\Omega(p)$ be a prescribed subset of $(\Z/p\Z)^n$. For $m\in \Z^n$, set
\[P(X;m):=\#\{p\leq X\mid m\!\!\mod{p}\in \Omega(p)\},\]
\[P(X)=\sum_{p\leq X} \#\Omega(p) p^{-n}.\]
Let $\cB$ be a box in $\mathbb{R}^n$ whose sides are parallel to the coordinate planes which has minimum width
$W(\cB)$ and volume $V(\cB)$. Then, the large sieve inequality states that if $W(\cB)\geq X^2$,  then
    \begin{equation}\label{large sieve ineq}\sum_{\cB\cap \Z^n} (P(X;m)-P(X))^2\ll V(\cB)P(X).\end{equation}
    We refer to the proof of \cite[Lemma 1]{duke1997elliptic} for the standard references on this theme.

    \par Taking $n=4$, let us define $\Omega(p)$ to be empty if $p=2,3$ or $p\not \equiv d\mod{\ell}$. For $p>3$ such that $p\equiv d\mod{\ell}$, set 
\[\Omega(p):=\{(r_1, s_1, r_2, s_2)\in (\Z/p\Z)^4\mid 4r_i^3+27s_i^2\neq 0\text{ and 
 } a_{r_i, s_i}(p)\equiv t_i\mod{\ell}\text{ for }i=1,2\}.\]
 For primes $p>3$, with $p\equiv d\mod{\ell}$, it follows from Theorem \ref{deuring thm} (or \cite[Lemma 2]{duke1997elliptic}) that
 \[\begin{split} &\# \{(r_1, s_1, r_2, s_2)\in (\Z/p\Z)^4\mid 4r_i^3+27s_i^2\neq 0, a_{r_i, s_i}(p)=a_i\text{ for }i=1,2\} \\ =& \left(\frac{p-1}{2}\right)^2 H(4p-a_1^2)H(4p-a_2^2).\end{split}\]
 We find that 
 \begin{equation}\label{H estimate delta}\begin{split}\#\Omega(p)=& \left(\frac{p-1}{2}\right)^2\prod_{i=1}^2\left(\sum_{a_i\equiv t_i\mod{\ell}}H(4p-a_i^2)\right).\\
 =& \delta p^4+O(\ell p^{7/2}),
 \end{split}\end{equation}
 where the second equality follows from \cite[Lemma 3]{duke1997elliptic}, which states that
 \[\left(\sum_{a_i\equiv t_i\mod{\ell}}H(4p-a_i^2)\right)=2\delta(t_i, d, \ell)p+O(\ell p^{1/2}).\] 
 Recall that $P(X)=\sum_{p\leq X} \#\Omega(p) p^{-4}$, and therefore, from the estimate \eqref{H estimate delta}, we find that 
 \begin{equation}\label{est 1}\begin{split}
 P(X)= & \sum_{p\equiv d \mod{\ell}; p\leq X} \delta + O\left(\ell \sum_{p\equiv d \mod{\ell}; p\leq X} p^{-1/2}\right),\\
 =& \delta \pi(X,d,\ell)+O(X^{1/2}).
 \end{split}\end{equation}
 The second estimate is the same as that of \cite[p.817, l.\ 4]{duke1997elliptic}.
Let $(E_1, E_2)\in \cC(X)\times \cC(X)$, with minimal Weierstrass equation 
\[E_i: y^2=x^3+r_i x +s_i.\]
Then, it is easy to see that 
\begin{equation}\label{est 2}P(X; (r_1, s_1, r_2, s_2))=\pi_{E_1, E_2} (X, t_1, t_2, d, \ell)+O(\op{log} X).\end{equation}
 We take 
 \[\cB:=\{(r_1, s_1, r_2, s_2)\in \mathbb{R}^4\mid |r_i|\leq X^2\text{ and } |s_i|\leq X^3\},\]
 and from \eqref{large sieve ineq}, \eqref{est 1} and \eqref{est 2}, we deduce that
 \[\begin{split}& \frac{1}{\# \mathcal{C}(X)^2}\sum_{E_1\in \mathcal{C}(X)}\sum_{E_2\in \mathcal{C}(X)}\left(\pi_{E_1, E_2}(X, t_1, t_2, d, \ell)-\delta\pi(X, d, \ell)\right)^2 
 \\ 
 \ll & \frac{1}{X^{10}}\sum_{\cB\cap \Z^n} (P(X;m)-P(X))^2 
 \ll P(X) 
 \ll  X.\end{split}\] 
 This proves the result.

\end{proof}

\begin{proposition}\label{T_A density 0}
    Let $A_{/\Q}$ be an elliptic curve, then, 
    \[\frac{\#\mathcal{T}_A(X)}{\# \mathcal{C}(X)}=O\left(\frac{\op{log}(X)}{\sqrt{X}}\right).\] In particular, the set $\cT_A$ has density $0$ in $\mathcal{C}$.
\end{proposition}
\begin{proof}
    We choose a triple $w=(a,b, d)\in \F_\ell^3$ such that $a\neq \pm b$ and $d\neq 0$. For any pair $(E_1,E_2)\in \mathcal{T}_A\times \mathcal{T}_A$, we find that \eqref{t_p(A)=pm t_p(E)} implies that the relation $t_p(E_1)=\pm t_p(E_2)$ holds. In particular, $(t_p(E_1), t_p(E_2))\neq (a,b)$. In particular, this implies that for $(E_1, E_2)\in \cT_A\times \cT_A$, 
    \[\pi_{E_1, E_2}(X, a, b, d,\ell)=0.\]
\par Invoking Proposition \ref{main prop 3.1}, we find that 
 \[\frac{1}{\# \mathcal{C}(X)^2}\sum_{E_1\in \mathcal{C}(X)}\sum_{E_2\in \mathcal{C}(X)}\left(\pi_{E_1, E_2}(X, a, b, d,\ell)-\delta\pi(X, d, \ell)\right)^2\leq CX.\]
 In particular, this implies that 
 \[\frac{1}{\# \mathcal{C}(X)^2}\sum_{E_1\in \mathcal{T}_A(X)}\sum_{E_2\in \mathcal{T}_A(X)}\left(\pi_{E_1, E_2}(X, a, b, d,\ell)-\delta\pi(X, d, \ell)\right)^2\leq CX.\]
 Since $\pi_{E_1, E_2}(X, a, b, d,\ell)=0$ for $(E_1, E_2)\in \cT_A\times \cT_A$, we find that 
  \[\left(\frac{\#\mathcal{T}_A(X)}{\# \mathcal{C}(X)}\right)^2\leq \frac{C}{\delta}\frac{X}{\pi(X,d,\ell)^2}=O\left(\frac{\op{log}(X)^2}{X}\right),\] and thus, 
   \[\frac{\#\mathcal{T}_A(X)}{\# \mathcal{C}(X)}=O\left(\frac{\op{log}(X)}{\sqrt{X}}\right).\]
   Since $\lim_{X\rightarrow \infty}\frac{\op{log}(X)}{\sqrt{X}}=0$, the result follows.
\end{proof}

\begin{proof}[Proof of Theorem \ref{main thm 1}/ Theorem \ref{our main result}]
    According to Proposition \ref{T K l prop}, every elliptic curve $E\in \cT_{K, \ell}$, $(E, K, \ell)$ satisfies (DS) and the conditions of Theorem \ref{mazur rubin criterion} are satisfied. Thus, it suffices to show that $\cT_{K, \ell}$ has density $1$. Let $L$ denote $\tilde{K}(\mu_\ell)$ and $\mathcal{D}_{K, \ell}$ be the set of elliptic curves $E_{/\Q}$ such that 
    \begin{enumerate}
    \item $\bar{\rho}_E'$ is surjective, 
    \item $\tilde{K}(\mu_\ell)$ contains $\Q(\bar{\rho}_E')$.
    \end{enumerate}

    Note that $\cT_{K, \ell}\cup \mathcal{D}_{K, \ell}$ consists of all elliptic curves $E_{/\Q}$ for which $\bar{\rho}_E'$ is surjective. By the main result of \cite{duke1997elliptic} discussed at the start of this section, their union has density $1$. Therefore, it suffices to show that $\mathcal{D}_{K, \ell}$ has density $0$. Let $F$ be a subfield of $L$ and $\mathcal{D}_F$ be the set of elliptic curves $E_{/\Q}$ such that
    \begin{enumerate}
    \item $\bar{\rho}_E'$ is surjective, 
    \item $\Q(\bar{\rho}_E')=F$.
    \end{enumerate}
    There are finitely many number fields $F$ contained in $L$. Therefore, it suffices to show that $\mathcal{D}_F$ has density $0$. Without loss of generality, we may assume that $\mathcal{D}_F$ is nonempty. Suppose that $A_1, A_2\in \mathcal{D}_F$. Then, since 
    \[\Q(\bar{\rho}_{A_1}')= \Q(\bar{\rho}_{A_2}')=F\]
    it follows that
    $\bar{\rho}_{A_1}'=\eta\circ \bar{\rho}_{A_2}'$, where $\eta\in \op{Aut}\left(\op{GL}_2'(\F_\ell)\right)$. Suppose that $A_2$ and $A_3$ are elliptic curves such that $\bar{\rho}_{A_1}'=\eta_i\circ \bar{\rho}_{A_i}'$ for $i=2,3$, where $\eta_i\in \op{Aut}\left(\op{GL}_2'(\F_\ell)\right)$. If $\eta_2=\eta_3$, then, $\bar{\rho}_{A_2}'\simeq \bar{\rho}_{A_3}'$. Note that the group $\op{Aut}\left(\op{GL}_2'(\F_\ell)\right)$ is finite. Therefore, we find that $\mathcal{D}_F$ is contained in a finite union $\bigcup_{i=1}^n \mathcal{T}_{A_i}$. Proposition \ref{T_A density 0} asserts that $\cT_{A}$ has density $0$ for any elliptic curve $A_{/\Q}$. Being a finite union of density $0$ sets, it follows that $\mathcal{D}_F$ has density $0$ as well. Since $\mathcal{D}_{K, \ell}$ is a finite union of sets of the form $\mathcal{D}_{F}$, it follows that $\mathcal{D}_{K, \ell}$ has density $0$, and therefore, $\cT_{K,\ell}$ has density $1$. This proves the result.
\end{proof}

\begin{proof}[Proof of Theorem \ref{main thm pos density}]
    The result is an immediate consequence of Theorem \ref{main thm 2.6}. 
\end{proof}
\section{Hilbert's Tenth Problem}\label{s 4}
\par In this section, we prove Theorem \ref{main thm 2}. Let us first recall the notion of an integrally diophantine extension of number fields, in the sense of \cite[Section 1.2]{shlapentokh2007hilbert}.
Let $A$ be a commutative ring with identity, and $A^n$ be the free $A$-module of rank $n$, consisting of tuples $a=(a_1, \dots, a_n)$ with entries in $A$. Let $m$ and $n$ be positive integers and let $a=(a_1, \dots, a_n)\in A^n$ and $b=(b_1, \dots, b_m)\in A^m$. Denote by $(a,b)\in A^{n+m}$ the tuple $(a_1, \dots, a_n, b_1, \dots, b_m)$. Given a finite set of polynomials, $F_1,\dots, F_k$, we set 
\[\cF(a; F_1, \dots, F_k):=\{b\in A^m\mid F_i(a,b)=0\text{ for all }i=1,\dots, k\}.\]
\begin{definition}\label{diophantine subset}A subset $S$ of $A^n$ is a \emph{diophantine subset} of $A^n$ if for some $m\geq 1$, there are polynomials $F_1, \dots, F_k\in A[x_1, \dots, x_n, y_1, \dots, y_m]$ such that $S$ consists of all $a\in A^n$ for which the set $\cF(a; F_1, \dots, F_k)$ is nonempty.
\end{definition}

\begin{definition}
An extension of number fields $L/K$ is said to be \emph{integrally diophantine} if $\cO_K$ is a diophantine subset of $\cO_L$.
\end{definition}
Let $L', L$ and $K$ be number fields such that 
\[L'\supseteq L\supseteq K.\]Suppose that $L'/L$ and $L/K$ are integrally diophantine extensions. Then, it is a well known fact that $L'/K$ is an integrally diophantine extension. Indeed, this is a special case of \cite[Theorem 2.1.15]{shlapentokh2007hilbert}.
\par We recall a conjecture of Denef and Lipchitz \cite{denef1978diophantine}.
\begin{conjecture}[Denef-Lipchitz]\label{denef lipchitz conjecture}
For any number field $L$, $L/\Q$ is an integrally diophantine extension. 
\end{conjecture}
The celebrated result of Matiyasevich proves that Hilbert's Tenth Problem has a negative answer over $\Z$. If $L$ is a number field for which Conjecture \ref{denef lipchitz conjecture} holds, then Hilbert's Tenth Problem has a negative answer for $\cO_L$.  The Conjecture \ref{denef lipchitz conjecture} above
is known to hold for various families of number fields $L$.
\begin{itemize}
    \item The conjecture holds when $L$ is either totally real or a quadratic extension of a totally real number field (cf. \cite{denef1980diophantine,denef1978diophantine}), 
    \item $L$ has exactly one complex place (cf. \cite{pheidas1988hilbert, shlapentokh1989extension, videla16decimo}),
    \item $L/\Q$ is abelian (cf. \cite{shapiro1989diophantine}),
    \item $[L:\Q]=4$, $L$ is not totally real and $L/\Q$ has a proper intermediate field (cf. \cite{denef1978diophantine}).
\end{itemize}

\par Poonen \cite{poonen2002using}, Cornelissen-Pheidas-Zahidi \cite{cornelissen2005division} and Shlapentokh \cite{shlapentokh2008elliptic} subsequently established new criteria for Conjecture \ref{denef lipchitz conjecture} to hold, stated in terms of the stability of the Mordell-Weil rank of an elliptic curve. The following criterion of Shlapentokh relates diophantine stability for positive rank elliptic curves to Hilbert's Tenth Problem for number rings.
\begin{theorem}[Shlapentokh]\label{shlap 1}
Let $L/K$ be an extension of number fields and suppose that there exists an elliptic curve $E_{/K}$ such that $\op{rank}E(L)=\op{rank}E(K)>0$. Then, $L/K$ is integrally diophantine. In particular, if Hilbert's Tenth Problem has a negative answer for $\cO_K$, then, Hilbert's Tenth Problem has a negative answer for $\cO_L$. 
\end{theorem}

\begin{proof}[Proof of Theorem \ref{main thm 2}]
    By Theorem \ref{main thm pos density}, there exists an elliptic curve $E_{/\Q}$ with positive rank which is diophantine stable at $\ell$. The result is then an immediate consequence of Theorem \ref{shlap 1}.
\end{proof}

\bibliographystyle{alpha}
\bibliography{references}

\begin{thebibliography}{ABHS25}

\bibitem[ABHS25]{alpoge2025rank}
Levent Alp{\"o}ge, Manjul Bhargava, Wei Ho, and Ari Shnidman.
\newblock Rank stability in quadratic extensions and {H}ilbert's tenth problem for the ring of integers of a number field.
\newblock {\em arXiv preprint arXiv:2501.18774}, 2025.

\bibitem[Bru92]{brumer1992average}
Armand Brumer.
\newblock The average rank of elliptic curves {I}.
\newblock {\em Inventiones mathematicae}, 109(1):445--472, 1992.

\bibitem[BS14]{BhargavaSkinner}
Manjul Bhargava and Christopher Skinner.
\newblock A positive proportion of elliptic curves over {$\Bbb Q$} have rank one.
\newblock {\em J. Ramanujan Math. Soc.}, 29(2):221--242, 2014.

\bibitem[Cox22]{cox2022primes}
David~A Cox.
\newblock {\em Primes of the Form $x^2+ ny^2$: {F}ermat, {C}lass Field Theory, and {C}omplex Multiplication. with Solutions}, volume 387.
\newblock American Mathematical Soc., 2022.

\bibitem[CPZ05]{cornelissen2005division}
Gunther Cornelissen, Thanases Pheidas, and Karim Zahidi.
\newblock Division-ample sets and the {D}iophantine problem for rings of integers.
\newblock {\em Journal de th{\'e}orie des nombres de Bordeaux}, 17(3):727--735, 2005.

\bibitem[Den80]{denef1980diophantine}
Jan Denef.
\newblock Diophantine sets over algebraic integer rings. {II}.
\newblock {\em Transactions of the American Mathematical Society}, 257(1):227--236, 1980.

\bibitem[DL78]{denef1978diophantine}
Jan Denef and Leonard Lipshitz.
\newblock Diophantine sets over some rings of algebraic integers.
\newblock {\em Journal of the {L}ondon {M}athematical {S}ociety}, 2(3):385--391, 1978.

\bibitem[Duk97]{duke1997elliptic}
William Duke.
\newblock Elliptic curves with no exceptional primes.
\newblock {\em Comptes rendus de l'{A}cad{\'e}mie des sciences. S{\'e}rie 1, Math{\'e}matique}, 325(8):813--818, 1997.

\bibitem[Gal46]{galois1846works}
{\'E}variste Galois.
\newblock {OE}mathematic works.
\newblock {\em Journal of Pure and Applied Mathematics}, 11:381--444, 1846.

\bibitem[Jon10]{jones2010almost}
Nathan Jones.
\newblock Almost all elliptic curves are {S}erre curves.
\newblock {\em Transactions of the American Mathematical Society}, 362(3):1547--1570, 2010.

\bibitem[Kai23]{kai2023linear}
Wataru Kai.
\newblock Linear patterns of prime elements in number fields.
\newblock {\em arXiv preprint arXiv:2306.16983}, 2023.

\bibitem[KP24]{koymans2024hilbert}
Peter Koymans and Carlo Pagano.
\newblock Hilbert's tenth problem via additive combinatorics.
\newblock {\em arXiv preprint arXiv:2412.01768}, 2024.

\bibitem[LW17]{lawson2017vanishing}
Tyler Lawson and Christian Wuthrich.
\newblock Vanishing of some {G}alois cohomology groups for elliptic curves.
\newblock In {\em Elliptic Curves, Modular Forms and Iwasawa Theory: In Honour of John H. Coates' 70th Birthday, Cambridge, UK, March 2015}, pages 373--399. Springer, 2017.

\bibitem[MRL18]{mazur2018diophantine}
Barry Mazur, Karl Rubin, and Michael Larsen.
\newblock Diophantine stability.
\newblock {\em American {J}ournal of {M}athematics}, 140(3):571--616, 2018.

\bibitem[Phe88]{pheidas1988hilbert}
Thanases Pheidas.
\newblock {H}ilbert’s tenth problem for a class of rings of algebraic integers.
\newblock {\em Proceedings of the American Mathematical Society}, 104(2):611--620, 1988.

\bibitem[Poo02]{poonen2002using}
Bjorn Poonen.
\newblock Using elliptic curves of rank one towards the undecidability of {H}ilbert’s tenth problem over rings of algebraic integers.
\newblock In {\em International Algorithmic Number Theory Symposium}, pages 33--42. Springer, 2002.

\bibitem[PS21]{prasad2021relating}
Dipendra Prasad and Sudhanshu Shekhar.
\newblock Relating the {T}ate--{S}hafarevich group of an elliptic curve with the class group.
\newblock {\em Pacific Journal of Mathematics}, 312(1):203--218, 2021.

\bibitem[Shl89]{shlapentokh1989extension}
Alexandra Shlapentokh.
\newblock Extension of {H}ilbert's tenth problem to some algebraic number fields.
\newblock {\em Communications on Pure and Applied Mathematics}, 42(7):939--962, 1989.

\bibitem[Shl07]{shlapentokh2007hilbert}
Alexandra Shlapentokh.
\newblock {\em {H}ilbert's tenth problem: {D}iophantine classes and extensions to global fields}.
\newblock Number~7. Cambridge University Press, 2007.

\bibitem[Shl08]{shlapentokh2008elliptic}
Alexandra Shlapentokh.
\newblock Elliptic curves retaining their rank in finite extensions and {H}ilbert’s tenth problem for rings of algebraic numbers.
\newblock {\em Transactions of the American Mathematical Society}, 360(7):3541--3555, 2008.

\bibitem[SS89]{shapiro1989diophantine}
Harold~N Shapiro and Alexandra Shlapentokh.
\newblock Diophantine relationships between algebraic number fields.
\newblock {\em Communications on Pure and Applied Mathematics}, 42(8):1113--1122, 1989.

\bibitem[Sut16]{sutherland2016computing}
Andrew~V Sutherland.
\newblock Computing images of {G}alois representations attached to elliptic curves.
\newblock In {\em Forum of Mathematics, Sigma}, volume~4, page~e4. Cambridge University Press, 2016.

\bibitem[Vid]{videla16decimo}
C~Videla.
\newblock Sobre el d{\'e}cimo problema de {H}ilbert.
\newblock {\em Atas da Xa Escola de Algebra, Vitoria, ES, Brasil. Colecao Atas}, 16:95--108.

\end{thebibliography}
\end{document}